\documentclass[letterpaper, 10 pt, conference]{ieeeconf}  

\IEEEoverridecommandlockouts                              
\usepackage{ntheorem}
\usepackage{hyperref}
\usepackage{textcomp}
\let\labelindent\relax
\usepackage{enumitem}
\usepackage{graphicx}
\usepackage{epstopdf}
\usepackage{float}
\usepackage[nocomma]{optidef}
\usepackage{dsfont}
\usepackage{times}
\usepackage{bm}
\usepackage{fancyhdr,graphicx,amsmath,amssymb}
\usepackage[noend]{algpseudocode}
\usepackage{bbm}
\usepackage{enumerate}
\usepackage{graphicx}
\usepackage{multirow}
\usepackage{caption}
\usepackage{subcaption}

\usepackage{adjustbox}

\usepackage{pgfplotstable}
\usepackage{pgfplots}
\pgfplotsset{
	table/search path={plot_figures},
}

\usepackage[percent]{overpic}

\usepackage{epsfig} 

\newtheorem{theorem}{Theorem}
\newtheorem{lemma}[theorem]{Lemma}
\newtheorem{proposition}[theorem]{Proposition}

\newtheorem{remark}{Remark}
\usepackage{booktabs}

\usepackage{pgfplotstable}
\usepackage{pgfplots}
\usepackage{cite}
\usepackage{algorithm}
\usepackage{algpseudocode}
\pgfplotsset{
	table/search path={plot_figures},
}

\allowdisplaybreaks

\title{\LARGE \bf
 On Acceleration of Gradient-Based Empirical Risk Minimization using Local Polynomial Regression
}

\author{Ekaterina Trimbach, Edward Duc Hien Nguyen, and C\'esar A. Uribe   
	\thanks{ET (\textit{trimbachkatya@gmail.com}) is with the École Polytechnique Fédérale de Lausanne and Moscow Institute of Physics and Technology; EN and CAU (\textit{\{en18,cauribe\}@rice.edu}) are with the Department of Electrical and Computer Engineering, Rice University, Houston, TX. EN is supported by a training fellowship from the Gulf Coast Consortia, on the NLM Training Program in Biomedical Informatics \& Data Science~(T15LM007093). This work was partially funded by \mbox{ARPA-H} Strategic Initiative Seed Fund~\#916012.}%
}

\begin{document}
	
	\maketitle
	\thispagestyle{empty}
	\pagestyle{empty}
	
	\begin{abstract} We study the acceleration of the Local Polynomial Interpolation-based Gradient Descent method (LPI-GD) recently proposed for the approximate solution of empirical risk minimization problems (ERM). We focus on loss functions that are strongly convex and smooth with condition number $\sigma$. We additionally assume the loss function is $\eta$-Hölder continuous with respect to the data. The oracle complexity of LPI-GD is $\tilde{O}\left(\sigma m^d \log(1/\varepsilon)\right)$ for a desired accuracy $\varepsilon$, where $d$ is the dimension of the parameter space, and $m$ is the cardinality of an approximation grid. The factor $m^d$ can be shown to scale as $O((1/\varepsilon)^{d/2\eta})$. LPI-GD has been shown to have better oracle complexity than gradient descent (GD) and stochastic gradient descent (SGD) for certain parameter regimes. We propose two accelerated methods for the ERM problem based on LPI-GD and show an oracle complexity of $\tilde{O}\left(\sqrt{\sigma} m^d \log(1/\varepsilon)\right)$. Moreover, we provide the first empirical study on local polynomial interpolation-based gradient methods and corroborate that LPI-GD has better performance than GD and SGD in some scenarios, and the proposed methods achieve acceleration. 
	\end{abstract}
	
	\section{Introduction}\label{sec:intro}
	
	Empirical Risk Minimization (ERM) has become a central component in the modeling of control and machine learning problems, e.g., classification, regression, inference~\cite{lugosi1995nonparametric,vapnik1992principles,chaudhuri2011differentially,donini2018empirical}. Alongside strong and broad modeling capabilities, ERM has available efficient tools with provable performance guarantees for the computation of its approximate solutions, for example, gradient descent (GD), stochastic gradient descent (SGD) and variations~\cite{nemirovski2009robust,dekel2012optimal,johnson2013accelerating,Nesterov_first_order_method}. 
	
	The ERM problem can be described as follows. For \mbox{$d, n \in \mathbb{N}$}, assume that there are $n$ samples of $d$-dimensional data points in denoted as $\mathcal{D}=\{x^{(i)} \in \left[h^{\prime}, 1-h^{\prime}\right]^{d}: i \in[n]\}$,  where $h^{\prime}>0$ is the range of each feature, and each train sample belongs to the hypercube $\left[h^{\prime}, 1-h^{\prime}\right]^{d}$. Our goal is to find a solution of the following optimization problem:
\begin{equation}\label{eq:gen_erm_problem}
F^{\star} = \min _{\theta \in \mathbb{R}^{p}} \left[ F(\theta)  \triangleq \frac{1}{n} \sum_{i=1}^{n} f\big(x^{(i)} ; \theta\big) \right]
\end{equation}
where $f: \mathbb{R}^{p} \rightarrow \mathbb{R}$ is the loss function. \textit{We are interested in finding approximate solutions $\hat \theta$ such that \mbox{$F(\hat \theta) - F^{\star} \leq \varepsilon$}, for some desired accuracy $\varepsilon>0$.}
	
	First-order methods such as gradient descent achieve linear convergence rates dependent on $\sigma$. However, their oracle complexity (i.e., the number of computations of gradients) can become prohibitively expensive when $n$ is large. For example, for a strongly convex and smooth loss function, gradient descent requires $O(n\log(1/\varepsilon))$ oracle calls to reach the desired accuracy $\varepsilon>0$. Intuitively, at every iteration, one needs to evaluate the gradient of the loss function on the $n$ data points. In contrast, like SGD, stochastic variants evaluate the gradient on a single randomly selected point, i.e., only one oracle call. While the number of gradient computations is reduced, this comes with a slower convergence rate, usually $O(1/t)$, where $t$ is the number of iterations, i.e., an oracle complexity of $O(1/\varepsilon)$. The complexity advantages of SGD become clear when considering statistical precision of the order $n^{-1/2}$ for large~$n$.
	
	The geometry of the loss function plays a crucial role in the iteration complexity of GD and SGD. If the loss function is assumed to be strongly convex and smooth, the iteration complexity depends on the condition number given by the geometry of the function. Traditionally, the geometry of the function is defined with respect to the optimization variable $\theta$. Recently in~\cite{LPI_GD_alg}, the authors proposed, for the first time, to exploit the smoothness of the loss function with respect to the data $x^{(i)}$. In~\cite{LPI_GD_alg}, the authors showed that under appropriate conditions, the smoothness of the loss function with respect to the data could lead to better oracle complexity of first-order methods for ERM problems. The main result in~\cite{LPI_GD_alg} shows regimes of the ERM problem where the proposed Gradient-Based Empirical Risk Minimization using Local Polynomial Regression (LPI-GD) achieves an oracle complexity of $O(n^\delta)$, for a small tunable parameter $\delta$. This complexity is provably smaller than the oracle complexity of SGD (i.e., $O(n^{1/2})$) when considering a statistical precision of $\varepsilon = \Theta(n^{-1/2})$.
	
	\begin{table}[t]
\centering
\caption{\textbf{Oracle complexity of GD, SGD, and LPI-GD~\cite{LPI_GD_alg}.} On upper bound on the oracle complexity of GD, SGD, and LPI-GD first-order methods with a desired accuracy $\varepsilon = \Theta(n^{-1/2})$, parameter dimension $p=O(\text{poly}(n))$, and data dimension $d = O(\log \log (n))$.}
  \begin{adjustbox}{max width=1\textwidth}
\begin{tabular}{  c  c  } 
 \toprule
 \textbf{ Algorithm }  & \textbf{Oracle Complexity} \\
 \midrule
  GD & $O(n\log(n))$ \\
  SGD & $O(\sqrt{n})$ \\
  LPI-GD & $O(n^\delta)$ \\
 \bottomrule
\end{tabular}
\end{adjustbox}
\label{table:complexities}
\end{table}

Table~\ref{table:complexities} shows the oracle complexity of LPI-GD compared to SGD and GD. LPI-GD has better oracle complexity under appropriate assumptions and parameter regimes~\cite{LPI_GD_alg}, i.e., small dimension and large smoothness with respect to the data. This sets an initial theoretical observation about the use of smoothness with respect to the data to reduce the oracle complexity of first-order methods. 

\textit{The main contribution of this paper} is to build upon the local polynomial regression technique for gradient interpolation and propose two accelerated methods that further reduce the oracle complexity of finding approximate solutions to ERM problems for smooth and strongly convex losses. The main contributions of this paper are threefold:
\begin{itemize}
    \item We propose a Catalyst-based algorithm that provably accelerates LPI-GD using the desired accuracy for the gradient interpolation prescribed by LPI-GD.
    \item We propose an accelerated method based on inexact oracle analysis that accelerates LPI-GD and describes the required accuracy of the gradient interpolation.
    \item We provide the first empirical study of LPI-GD and compare its performance with the two proposed accelerated variants.
\end{itemize}

This paper is organized as follows. Section~\ref{sec:assumpt} recalls the local polynomial interpolation and the LPI-GD method. Sections~\ref{sec:algo_catalyst}, and~\ref{sec:algo} present the Catalyst-based and inexact oracle-based accelerated methods respectively. Sections~\ref{sec:num}, and~\ref{sec:diss} present the preliminary numerical analysis and discussion. Section~\ref{sec:conclusion} finalizes with conclusions and future work.

\noindent\textbf{Notation:} We denote  $[n] \triangleq\{1, \ldots, n\}$ for all $n \in \mathbb{N}$. For any $d$-tuple  non-negative integers $s = (s_1, s_2, ..., s_d)$ denote $|s|=s_{1}+\cdots+s_{d}, s !=s_{1} ! \cdots s_{d} !, $
$\text { and } x^{s}=x_{1}^{s_{1}} \cdots x_{d}^{s_{d}} \text { for any } x \in \mathbb{R}^{d}$.
 
\section{Preliminaries}\label{sec:assumpt}

This section states the assumptions on the function classes that we will use for our analysis. Moreover, we will recall the LPI-GD algorithm and its theoretical performance.

\subsection{Assumptions}\label{sec:assum}
\subsubsection{\textbf{Smoothness in the parameter}}\label{as:1}
For any $\theta_{1}, \theta_{2} \in \mathbb{R}^{p}$ there exists $L_{1}>0$ such that for any fixed $x \in[0,1]^{d}$, the gradient $\nabla_{\theta} f(x ; \cdot): \mathbb{R}^{p} \rightarrow \mathbb{R}^{p}$ exists and is $L_{1}$-Lipschitz continuous as a function of the parameter vector:
\begin{align}\label{eq:L_1_def}
\left\|\nabla_{\theta} f\left(x ; \theta_{1}\right)-\nabla_{\theta} f\left(x ; \theta_{2}\right)\right\|_{2} \leq L_{1}\left\|\theta_{1}-\theta_{2}\right\|_{2}.
\end{align}
\subsubsection{\textbf{Smoothness in the data}}\label{as:2}
For any $y_{1}, y_{2} \in[0,1]^{d}$ there exist $\eta>0$ and $L_{2}>0$ such that $ \forall \vartheta \in \mathbb{R}^{p}$ and  $\forall i \in[p]$, the $i$'th partial derivative with respect to $\theta_{i}$ denoted $g_{i}(\cdot ; \vartheta) \triangleq$ $\frac{\partial f}{\partial \theta_{i}}(\cdot ; \vartheta):[0,1]^{d} \rightarrow \mathbb{R}$, belongs to the $\left(\eta, L_{2}\right)$-Hölder class as a function of the data vector $x$, i.e.,
$$
 \quad\left|\nabla^{s} g_{i}\left(y_{1} ; \vartheta\right)-\nabla^{s} g_{i}\left(y_{2} ; \vartheta\right)\right| \leq L_{2}\left\|y_{1}-y_{2}\right\|_{1}^{\eta-l},
$$
where $l=\lceil\eta\rceil-1$\\
\subsubsection{\textbf{Strong convexity}}\label{as:3}
For all $\theta_{1}, \theta_{2} \in \mathbb{R}^{p}$ there exists \mbox{$\mu>0$} such that for all fixed $x \in[0,1]^{d}$, the map $f(x ; \cdot): \mathbb{R}^{p} \rightarrow$ $\mathbb{R}$ is $\mu$-strongly convex, i.e.,
$$
f\left(x ; \theta_{1}\right) \geq f\left(x ; \theta_{2}\right)+\nabla_{\theta} f\left(x ; \theta_{2}\right)^{\mathrm{T}}\left(\theta_{1}-\theta_{2}\right)+\frac{\mu}{2}\left\|\theta_{1}-\theta_{2}\right\|_{2}^{2}.
$$

Additionally, we define the condition number of the function $f$ as  $\sigma \triangleq L_{1} / \mu$.

{Assumptions \ref{as:1}, \ref{as:2}, \ref{as:3} are standard for convex optimization methods~\cite{nesterov2018lectures}. Classically,  optimization methods require Lipschitz continuous or continuously differentiable properties, but due to the fact that this Hölder class includes both Lipschitz continuous and continuously differentiable functions, \ref{as:2} allows deal with a larger function class.}

\subsection{LPI-GD Algorithm}

\begin{algorithm}[t!]
\caption{Local Polynomial Interpolation based Gradient Descent (LPI-GD)~\cite{LPI_GD_alg}.}\label{alg:LPI_GD}
\begin{algorithmic}
\Require $\theta^{0}$ (initial guess); $K$ (number of iterations); $\varepsilon>0$ (desired accuracy), $\{\delta_t \in (0,1)\}$ for $t \in T$, a sequence of allowable supremum norm approximation.
\For{$k = 1, \ldots, K$}
\State \textbf{1.} Compute $\widehat{\nabla F}^{(k)}(\theta_{k - 1})$ with accuracy $\delta_t$. 
\State \textbf{2.} Update: $\theta^{(k)} \leftarrow \theta^{(k-1)}-\frac{1}{L_{1}} \widehat{\nabla F}^{(k)}(\theta_{k - 1})$
\EndFor
\Return Output: $\theta_{K}$ (final estimate).
\end{algorithmic}
\label{algo:LPI-GD}
\end{algorithm}

LPI-GD follows the same structure as gradient descent, where at each iteration, the parameters are updated in the negative direction of the gradient, see Algorithm~\ref{algo:LPI-GD}. However, the critical difference is that LPI-GD builds an approximate gradient $\widehat{\nabla F}^{(k)}(\cdot)$ at each iteration $k$ using local polynomial regressions and exploiting the smoothness of the loss function with respect to the data.

The approximate gradient is computed as follows:
\begin{align} \label{def:approx_grad_F}
    \widehat{\nabla F}^{(k)} (\theta^{(k-1)})=\frac{1}{n} \sum_{j=1}^{n}\left[\hat{\phi}_{i}^{(k)}\big(x^{(j)}\big): i \in[p]\right]^{\mathrm{T}},
\end{align}
where  $x \in\left[h^{\prime}, 1-h^{\prime}\right]^{d}$, and $\hat{\phi}_{i}^{(k)}$ is a coordinate-wise approximator defined as
\begin{align}\label{def:hat_fi_i}
 \hat{\phi}_{i}^{(k)}(x)=\sum_{y \in \mathcal{G}_{m_{t}}} g_{i}\big(y ; \theta^{(k-1)}\big) w_{y, k}^{\star}(x),
\end{align}
with $g_{i}\left(\cdot ; \theta^{(k-1)}\right)=\frac{\partial f}{\partial \theta_{i}}\left(\cdot ; \theta^{(k-1)}\right):\left[h^{\prime}, 1-h^{\prime}\right]^{d} \rightarrow \mathbb{R}$ is the partial derivative and $w^{\star}$ is a set of interpolation weights. Importantly, the coordinate approximation functions are computed as weighted interpolations defined on an uniform grid
\begin{align*}
    \mathcal{G}_m \triangleq \{ u \in [0,1]^d: \forall i\in [d],u_im \in m \},
\end{align*}
which has cardinality $|\mathcal{G}_m| = m^d$. The weights $w^{\star}$ depend on the weight vectors $w_{y}(x)$ defined as
\begin{align*}
  &  \forall y \in \mathcal{G}_{m}, \forall x \in[h, 1-h]^{d}, w_{y}^{\star}(x)=\left[w_{y}(x)\right]_{1} \\
& w_{y}(x) =\frac{1}{(m h)^{d}}\left(\prod_{j=1}^{d} K\left(\frac{y_{j}-x_{j}}{h}\right)\right) B(x)^{-1} U\left(\frac{y-x}{h}\right).
\end{align*}
Interpolation is based on a kernel function $K(x)$, where, for any $x \in[h, 1-h]^{d}$, the symmetric matrix $B(x)$ is
\begin{align*}
&B(x) \triangleq\\
&\frac{1}{(m h)^{d}} \sum_{y \in \mathcal{G}_{m}} U\left(\frac{y-x}{h}\right) U\left(\frac{y-x}{h}\right)^{\mathrm{T}} \prod_{j=1}^{d} K\left(\frac{y_{j}-x_{j}}{h}\right),
\end{align*}
and, the vector function $U(u)$ is defined as
\begin{align*}
\forall u \in \mathbb{R}^{d}, \quad U(u) \triangleq\left[\frac{u^{s}}{s !}: s \in \mathbb{Z}_{+}^{d},|s| \leq l\right]^{\mathrm{T}}.
\end{align*}

For more details on this specific local polynomial regression see~\cite[Section 2.1]{LPI_GD_alg}. The key implication of the interpolation formulation described above is shown in the next informal theorem.
\begin{theorem}[Theorem 2.2 in~\cite{LPI_GD_alg}, simplified]\label{theo:inter}
    For any constant $\delta \in (0,1)$, under the smoothness and convexity assumptions presented in Section~\ref{sec:assumpt}, there exists a grid cardinality parameter $m >0$ such that the supremum norm between the local polynomial interpolator $\hat\phi(x)$ and the gradient $g(x)$ at a point $x$ is bounded by $\delta$, i.e.,
    \begin{align*}
        \sup_{x \in [h,1-h]^d} |\hat\phi(x) - g(x)| \leq \delta.
    \end{align*}
\end{theorem}


Under the local polynomial regression described above and Theorem~\ref{theo:inter}, Algorithm~\ref{alg:LPI_GD} is shown to have the following iteration complexity:
\begin{proposition}[Proposition 2.3 in
~\cite{LPI_GD_alg}, simplified]\label{prop:iter}
 The iteration complexity for Algorithm~\ref{alg:LPI_GD} for accuracy $\varepsilon>0$ and start point $\theta^{(0)}$ is :
 \begin{align*}
     K=\left[\sigma \log \left(\left({F\left(\theta^{(0)}\right)-F_{*}+\frac{p}{2 \mu}}\right)\big/{\varepsilon}\right)\right],
 \end{align*}
 with a maximum approximation error \mbox{$\delta_k^2  = \Theta(\varepsilon/p)$}. 
\end{proposition}

The authors in~\cite[Theorem~$2.4$]{LPI_GD_alg} show an estimate of $|\mathcal{G}_m| = m^d$ as
 \begin{align*}
     m^d & =  O \left( C_{\mu,L_1,L_2}(d,\eta)  \left({1}/{\varepsilon} \right)^{d/(2\eta)} \right),
 \end{align*}
 for a positive constant $C_{\mu,L_1,L_2}(d,\eta)$.  Thus, the oracle complexity of Algorithm~\ref{alg:LPI_GD} can be estimated as 
 \begin{align}\label{eq:total}
   O\left( \sigma C_{\mu,L_1,L_2}(d,\eta)  \left({1}/{\varepsilon} \right)^{d/(2\eta)} \log ({1}/{\varepsilon}) \right).
 \end{align}
 
\begin{remark}
Our main technical contributions will be to propose two accelerated methods with oracle complexities $O\left( \sqrt{\sigma} m^d \log ({1}/{\varepsilon}) \right)$. The term $m^d$ corresponding to the grid size cannot be improved as it is the minimum required to reach the desired accuracy at each iteration. Therefore, we focus on improving the explicit dependency on the condition number of the function, from $\sigma$ to $\sqrt{\sigma}$. Moreover, we will provide the first empirical analysis of local polynomial interpolation for approximate gradient methods.
\end{remark}

The following sections describe two approaches for accelerating the LPI-GD algorithm and their corresponding iteration and oracle complexities.

	\section{Catalyst Acceleration}\label{sec:algo_catalyst}

	Catalyst is one of the most used generic acceleration methods in machine learning for its ability to provably generically accelerate large classes of base algorithms~\cite{Catalyst_main_acc,catalyst_18,catalyst_15,Catalyst_DSGD,paquette2017catalyst}. It can provably achieve better iteration complexity of a base algorithm with a generic iterative approximation procedure. The main idea is to minimize a particular surrogate function at every iteration and warm start the next step on the approximate minimizer of the surrogate function of the previous iteration, see~\cite{catalyst_18,catalyst_15} for a detailed description. 
	
	Acceleration using Catalyst can be achieved when the surrogate function $h_k$ satisfying the following properties~\cite{Catalyst_main_acc}:
\begin{align*}
    \left(\mathcal{H}_{1}\right)\quad &h_{k} \text{ is } (\kappa+\mu)\text{-strongly convex}. \\
    \left(\mathcal{H}_{2}\right)\quad &\mathbb{E}[h_{k}(\theta)] \leq F(\theta)+\frac{\kappa}{2}\left\|\theta-z_{k-1}\right\|^{2} \\
    &\text{ for } \theta= \alpha_{k-1}\text{argmin}_{\theta} \{ F(\theta) \}+\left(1-\alpha_{k-1}\right) \theta_{k-1}. \\
    \left(\mathcal{H}_{3}\right)\quad &\forall \varepsilon_{k} \geq 0,~ \mathcal{M} \text{ can provide a point }\\
    &\theta_{k} : \mathbb{E}\left[h_{k}\left(\theta_{k}\right)-h_{k}^{\star}\right] \leq \varepsilon_{k},
\end{align*}
where $\mathcal{M}$ is the base method we are trying to accelerate, LPI-GD in our case.

Under the properties $\left(\mathcal{H}_{1}\right)$, $\left(\mathcal{H}_{2}\right)$, and $\left(\mathcal{H}_{3}\right)$, we can define a generic Catalyst acceleration method following~\cite{Catalyst_main_acc} as described in Algorithm~\ref{algo:cata}.

\begin{algorithm}[t!]
\caption{Catalyst Acceleration LPI-GD}\label{alg:catalyst_acc}
\begin{algorithmic}
\Require $\theta^{0}$ (initialization); $\mu$ (strong convexity constant); $\kappa$ (parameter for $h_{k}) ; K$ (number of iterations); $\left\{\varepsilon_{k}\right\}_{k=1}^\infty$ (sequence of approximation errors).
\Ensure $z_{0}=\theta^{(0)};~ q = \frac{\mu}{\mu+\kappa};~ \alpha_{0}=\sqrt{q}$.
\For{$k = 1, \ldots, K$}
\State \textbf{1.} Choose a surrogate $h_{k}$ satisfying $\left(\mathcal{H}_{1}\right),\left(\mathcal{H}_{2}\right)$ and calculate $\theta_{(k)}$ using Algorithm~\ref{alg:LPI_GD}, satisfying $\left(\mathcal{H}_{3}\right)  \text{ for } \varepsilon_{k};$ 
\State \textbf{2.} Find $\alpha_{k}\in(0,1)$ s.t. $\alpha_{k}^{2}=\left(1-\alpha_{k}\right) \alpha_{k-1}^{2}+q \alpha_{k} .$
\State \textbf{3.} Update the extrapolated sequence $z_{k}=\theta^{(k)}+\beta_{k}\left(\theta^{(k)}-\theta^{(k-1)}\right)$ 
    with $\beta_{k}=\frac{\alpha_{k-1}\left(1-\alpha_{k-1}\right)}{\alpha_{k-1}^{2}+\alpha_{k}}$.
\EndFor
\Return Output: $\theta_{K}$ (final estimate).
\end{algorithmic}
\label{algo:cata}
\end{algorithm}

The following theorem presents an explicit oracle and iteration complexity analysis of the Catalyst technique applied to the acceleration of Algorithm~\ref{alg:LPI_GD}. 

\begin{theorem}\label{theo:cata}
 Let $\{h_k(\theta)\}_{k\geq 0}$ defined as $h_k(\theta) = f(\theta) + \frac{\kappa}{2} \|  \theta - z_{k-1}  \|$ for $k\geq 0$, where $\kappa = L_1 -\mu$, and  $\varepsilon_{k}=O\left((1-1/3\sqrt{\sigma})^{k}\left(F\left(\theta_{0}\right)-F^{\star}\right)\right)$. Then, the total iteration complexity of Algorithm~\ref{alg:catalyst_acc} can be bounded as
\begin{align*}
&  O\Big(\sqrt{\sigma} \log \left(\frac{F\left(\theta_0\right)-F^{\star}}{\sigma \varepsilon}\right)  C_{L_1,L_h,L_2}(d,\eta) \times \\
 & \qquad \times \left(\frac{p }{\varepsilon_K} + \frac{ 2 L_1}{q^2} \right)^{\frac{d}{2\eta}} \log \left(\frac{1}{q^2}+\frac{ {p}/{(2 L_1)}}{\varepsilon_K}\right) \Big).
\end{align*}
\end{theorem}

\begin{proof}
In~\cite{Catalyst_main_acc}, the authors showed  that after running Algorithm~\ref{alg:catalyst_acc} for $k$ iterations, the following inequality holds:
\begin{align*}
    & F\left( \theta^{(k)}\right)-F^{\star} \leq\left(1-\frac{\sqrt{q}}{2}\right)^{(k)} \times\\
    &\left(2\left(F\left( \theta^{(0)}\right)-F^{\star}\right) +
     4 \sum_{j=1}^{k}\left(1-\frac{\sqrt{q}}{2}\right)^{-j}\left(\varepsilon_{j}+\frac{\varepsilon_{j}}{\sqrt{q}}\right)\right),
\end{align*}
where $q = 1/\sigma$. From the initial hypothesis we have that $h_k(\theta) = F(\theta) + \frac{\kappa}{2} \|  \theta - y_{k-1}  \|$. If we choose  $\varepsilon_{k}=O\left((1-1/\left(3\sqrt{\sigma}\right))^{k}\left(F\left(x_{0}\right)-F^{\star}\right)\right)$, then the number of outer iterations will be $K=O\left(\sqrt{\sigma} \log \left(\sigma {(F\left(\theta_0\right)-F^{\star})}/{ \varepsilon}\right)\right)$, see~\cite[Supplementary B.3]{Catalyst_main_acc}. Moreover, when  $h_{k}\left(\theta_{k-1}\right)-h_{k}^{\star}=O\left(\frac{\kappa}{\mu q} \varepsilon_{k-1}\right)=O\left(\varepsilon_{k-1} / q^{2}\right)$, the final accuracy is
$\varepsilon=O\left(\varepsilon_{K} / q\right)$~\cite[ Proposition 5]{Catalyst_main_acc}.

Setting $\mu_h = \mu + \kappa = L_1$ and $L_h = 2L_1 - \mu$ it follows from~\eqref{eq:total}, that the number of inner iteration required at each inner loop of Algorithm~\ref{alg:catalyst_acc} is
\begin{align}\label{eq:total_cata}
&\Big( \sigma_h C_{L_1,L_h,L_2}(d,\eta)  \left(\frac{p +2L_1(h\left(\theta^{(k - 1)}\right)-h_{*})}{\varepsilon_k} \right)^{d/(2\eta)} \times \nonumber \\
& \qquad \qquad \log \left(\frac{h\left(\theta^{(k - 1)}\right)-h_{*}+\frac{p}{2 L_1}}{\varepsilon_k}\right) \Big).
\end{align}
Note that in~\eqref{eq:total_cata} we have a prespecified accuracy $\varepsilon_k$, and the condition number is with respect to the surrogate function, i.e., $\sigma_{h} = \frac{L_1 + \kappa}{\mu + \kappa}$, with $\kappa = L_1 - \mu$, then $\sigma_{h} =  \frac{2L_1 + \mu}{L_1} $. Therefore the total number of oracle calls will be the aggregation of~\eqref{eq:total_cata} across the total number of outer loop iterations $K$, i.e.,
\begin{align*}
    &\sum_{k=1}^{K} \Big( \sigma_h C_{L_1,L_h,L_2}(d,\eta)  \left(\frac{p {+}2L_1(h\left(\theta^{(k {-} 1)}\right){-}h_{*})}{\varepsilon_k} \right)^{d/(2\eta)} \times \nonumber \\
& \qquad \qquad \log \left(\frac{h\left(\theta^{(k - 1)}\right)-h_{*}+\frac{p}{2 L_1}}{\varepsilon_k}\right) \Big) \Big).
\end{align*} 
Note that following $\mathcal{H}_{3}$ we have that $h\left(\theta^{(k - 1)}\right)-h_{*}\leq O\left(\varepsilon_{k-1} / q^{2}\right)$, thus,
\begin{align*}
    &\sum_{k=1}^{K} \Big( \sigma_h C_{L_1,L_h,L_2}(d,\eta)  \left(\frac{p {+}\frac{2  L_1\varepsilon_{k-1}}{q^2}}{\varepsilon_k} \right)^{\frac{d}{2\eta}} \log \left(\frac{\frac{\varepsilon_{k-1}}{q^2} {+}\frac{p}{2 L_1}}{\varepsilon_k}\right) \\
    & = O\Big(\sqrt{\sigma} \log \left(\frac{F\left(\theta_0\right)-F^{\star}}{\sigma \varepsilon}\right) \sigma_h C_{L_1,L_h,L_2}(d,\eta) \times \\
    & \qquad \times \left(\frac{p }{\varepsilon_K} + \frac{ 2 L_1}{q^2} \right)^{\frac{d}{2\eta}} \log \left(\frac{1}{q^2}+\frac{ \frac{p}{2 L_1}}{\varepsilon_K}\right) \Big), 
\end{align*}
where $\sigma_h = \frac{L_h}{m_h} = \frac{2L_1 - \mu}{L_1} < 2$
This completes the proof.
\end{proof}

Theorem~\ref{theo:cata} shows that the dependency on the condition number is reduced from $\sigma$ to $\sqrt{\sigma}$, which is traditionally understood as acceleration. Specifically, the total number of required oracle calls is
 \begin{align*}
   O\left( \sqrt{\sigma} C_{L_1,L_h,L_2}(d,\eta)  \left({1}/{\varepsilon} \right)^{d/(2\eta)} \log ({1}/{\varepsilon}) \right).
 \end{align*}

In the next section, we will explore another form of acceleration based on inexact oracles theory.

\section{Nesterov acceleration}\label{sec:algo}
\begin{algorithm}[t!]
\caption{Fast Gradient Method LPI-GD}\label{alg:nesterov_acc}
\begin{algorithmic}
\Require Define 
\begin{align*}
&d(\theta) = \frac{1}{2}\left\|\theta-\theta_{0}\right\|_{E}^{2}, 
H(\theta) = L d(\theta)+\\
&\sum_{i=0}^{k} \alpha_{i}\left[\left\langle g_{\delta, L, \mu}\left(\theta_{i}\right), \theta-\theta_{i}\right\rangle+\frac{\mu}{2}\left\|\theta-\theta_{i}\right\|_{E}^{2}\right]
\end{align*}
Define $A_{k}=\sum_{i=0}^{k} \alpha_{i}$, $\alpha_{0}={L}/{(L-\mu)}$, $(L-\mu) \alpha_{k+1}=A_{k} \mu+L$, $\tau_{k}=\frac{\alpha_{k+1}}{A_{k+1}}, k \geq 0$, and initialize $\theta_{0}$.
\For{$k = 1, \ldots, K$}
\State \textbf{1.} $\text {Obtain }\left(f_{\delta, L, \mu}\left(\theta_{k}\right), g_{\delta, L, \mu}\left(\theta_{k}\right)\right)$ 
\State \textbf{2.} Compute
\State \begin{small}
$y_{k}=\arg \min_{\theta}\left\{\left\langle g_{\delta, L, \mu}\left(\theta_{k}\right), \theta-\theta_{k}\right\rangle+\frac{L}{2}\left\|\theta-\theta_{k}\right\|_{E}^{2}\right\}$
\end{small}
\State \textbf{3.} $\text {Compute }$ 
$ z_{k}=\arg \min _{\theta}\left\{ H(\theta)\right\}$.
\State \textbf{4.} $\text { Define } \theta_{k+1}=\tau_{k} z_{k}+\left(1-\tau_{k}\right) y_{k}$
\EndFor
\Return Output: $\theta_{K}$ (final estimate).
\end{algorithmic}
\end{algorithm}

In this section, we study  Nesterov acceleration for LPI-GD~\cite{Nesterov_strong_acc}. Algorithm~\ref{alg:catalyst_acc} uses LPI-GD to minimize the surrogate function at each iteration, while the work in~\cite{Nesterov_strong_acc} uses the Fast Gradient Method algorithm with inexact oracles. In~\cite{Nesterov_strong_acc}, the authors assume that functions $f_{\delta, L, mu}(\theta)$ and $g_{\delta, L, mu}(\theta)$ exists, which are an approximation of the value of the function and the gradient at the point  $\theta$ respectively. Those functions should satisfy:
\begin{align}\label{eq:delta_def}
\frac{\mu}{2}\|\theta_1- \theta_2\|_{E}^{2} \leq G(\theta_1, \theta_2) \leq \frac{L}{2}\|\theta_1-\theta_2\|_{E}^{2}+\delta,
\end{align}
\begin{small}
where
\begin{align}\label{eq:G}
G(\theta_1, \theta_2) = f(\theta_1)-\left(f_{\delta, L, \mu}(\theta_2)+\left\langle g_{\delta, L, \mu}(\theta_2), \theta_1-\theta_2\right\rangle\right)
\end{align}
\end{small}
for all $\theta_1, \theta_2$ where $\delta \geq 0$ and $L \geq \mu \geq 0$.

We will show that the gradient approximation based on the local polynomial interpolation method presented in Section~\ref{sec:assumpt} has the properties in~\eqref{eq:delta_def} for a particular value $\delta$. Thus, $f_{\delta, L, \mu} = F(\theta)$ and $g_{\delta, L, \mu}(\theta) =  \widehat{\nabla F}(\theta)$. Next, we present our a technical lemma that shows that the approximate gradient computed using local polynomial interpolations is an inexact oracle in the sense of~\eqref{eq:delta_def}.


\begin{lemma}
The gradient approximation based on local polynomial interpolation where $g_{\delta, L, \mu}(\theta) =  \widehat{\nabla F}(\theta)$ is an inexact oracle in the sense of~\eqref{eq:G} with $\delta = p\left( 1 - \frac{1}{\sigma} \right)^2$ and $L = L_1^2 + 2 - \frac{\mu}{2}$.
\end{lemma}

\begin{proof}
Our task is to find $\delta>0$ such:
\begin{align*}
   &0 \leq f(x, \theta_1)-\left(f(x, \theta_2)+\left\langle g_{\delta, L}(\theta_2), \theta_1-\theta_2\right\rangle\right) \leq\\
    &\frac{L}{2}\|\theta_1 - \theta_2\|_{E}^{2}+\delta.
\end{align*}
Using  strongly convexity we have that
\begin{align*}
&\forall \theta_{1}, \theta_{1} \in \mathbb{R}^{p},\quad f\left(x ; \theta_{2}\right) \geq \\
&f\left(x ; \theta_{1}\right)+\nabla_{\theta} f\left(x ; \theta_{1}\right)^{\mathrm{T}}\left(\theta_{2}-\theta_{1}\right)+\frac{\mu}{2}\left\|\theta_{2}-\theta_{1}\right\|_{2}^{2},
\end{align*}
which can be rewritten as
\begin{align*}
&f(x, \theta_1) - f(x, \theta_2) \leq\\
&- \langle \nabla_{\theta} f\left(x ; \theta_{1}\right), \theta_{2}-\theta_{1} \rangle -\frac{\mu}{2}\left\|\theta_{2}-\theta_{1}\right\|_{2}^{2}
\end{align*}
Therefore,
\begin{align*}
&f(x, \theta_1) - f(x, \theta_2) - \langle \widehat{\nabla F}^{(t)}(\theta_2), \theta_1 - \theta_2 \rangle  \leq\\  
& \langle \widehat{\nabla F}^{(t)}(\theta_2) - \nabla_{\theta} f\left(x ; \theta_{1}\right), \theta_2 - \theta_1 \rangle  - \frac{\mu}{2}\left\|\theta_{2}-\theta_{1}\right\|_{2}^{2} = \\
&  \langle \widehat{\nabla F}^{(t)}(\theta_2) - \nabla_{\theta} f\left(x ; \theta_{2}\right), \theta_2 - \theta_1 \rangle +\\
&\langle\nabla_{\theta} f\left(x ; \theta_{2}\right) - \nabla_{\theta} f\left(x ; \theta_{1}\right), \theta_2 - \theta_1 \rangle - \frac{\mu}{2}\left\|\theta_{2}-\theta_{1}\right\|_{2}^{2}
\end{align*}
It follows from \cite[Eq. 60]{LPI_GD_alg}, that
\begin{align*}
&\langle \widehat{\nabla F}^{(t)}(\theta_2) - \nabla_{\theta} f\left(x ; \theta_{2}\right), \theta_2 - \theta_1 \rangle \leq \\
&\| \widehat{\nabla F}^{(t)}(\theta_2) - \nabla_{\theta} f\left(x ; \theta_{2}\right) \|_2^2 + \|  \theta_2 - \theta_1  \|_2^2 \\
&\leq (\sqrt{p} \delta_0)^2 + \left\|\theta_{2}-\theta_{1}\right\|_{2}^{2},
\end{align*}
and from \cite[Assumption 5]{LPI_GD_alg}, we obtain
\begin{align*}
&\langle\nabla_{\theta} f\left(x ; \theta_{2}\right) - \nabla_{\theta} f\left(x ; \theta_{1}\right), \theta_2 - \theta_1 \rangle  \leq\\
&\| \nabla_{\theta} f\left(x ; \theta_{2}\right) - \nabla_{\theta} f\left(x ; \theta_{1}\right) \|_2^2  + \left\|\theta_{2}-\theta_{1}\right\|_{2}^{2}\leq\\
&( L_1^2+1)\left\|\theta_{2}-\theta_{1}\right\|_{2}^{2}.
\end{align*}
Finally, we obtain
\begin{align*}
&f(x, \theta_1) - f(x, \theta_2) - \langle \widehat{\nabla F}^{(t)}(\theta_2), \theta_1 - \theta_2 \rangle \leq \\
&(L_1^2 + 2 - \frac{\mu}{2}) \left\|\theta_{2}-\theta_{1}\right\|_{2}^{2} + p \delta_0^2.
\end{align*}
Where $\delta_0 = \left(1 - \frac{1}{\sigma} \right)$.
\end{proof}

We are now ready to state our accelerated convergence result for Algorithm~\ref{alg:nesterov_acc}.

\begin{theorem}\label{theo:nesterov}
Let $\varepsilon>0$ be a desired accuracy, 
\begin{align*}
    \delta \leq  \frac{\varepsilon}{2\sqrt{p}}(1 {+} \sigma)^{-1} \quad \text{and} \quad K = 2\sqrt{\sigma}\log\big({L_1 \left\|\theta^{\star}{-}\theta_{0}\right\|_{2}^{2}}/{\varepsilon}\big),
\end{align*}
Then, the output of algorithm Algorithm~\ref{alg:nesterov_acc} has the following property: $F(\theta_K) - F^{\star} \leq \varepsilon$.
\end{theorem}

\begin{proof}
It follows from~\cite[Th. 7]{Nesterov_strong_acc} that at every iteration $k$ the gap between current point and exact solution is bounded as
\begin{align}\label{eq:in_nes}
& f\left(y_{k}\right)-f^{*} \leq  \min \left(\frac{4 L d\left(\theta^{*}\right)}{k^{2}}, L d\left(\theta^{*}\right) \exp \left(-\frac{k}{2} \sqrt{\frac{\mu}{L}}\right)\right) \nonumber \\
&+\min \left(\left(\frac{1}{3} k+2.4\right),\left(1+\sqrt{\frac{L}{\mu}}\right)\right) \delta.
\end{align}

Therefore, to guarantee a suboptimality gap of \mbox{$\varepsilon>0$}, we can bound each factor in~\eqref{eq:in_nes} by ${\varepsilon}/{2}$. On the one hand, in order for $L d\left(\theta^{*}\right) \exp \left(-\frac{k}{2} \sqrt{\frac{\mu}{L}}\right) \leq \frac{\varepsilon}{2}$, then $k \geq 2\sqrt{\frac{L}{\mu}}\ln\left(\frac{2L d(\theta^{\star})}{\varepsilon}\right)$. On the other hand, in order to guarantee that $\delta\left(1+\sqrt{\frac{L}{\mu}}\right) \leq \frac{\varepsilon}{2}$, it is sufficient to set \mbox{$\delta \leq  \frac{\varepsilon}{2\sqrt{p}}(1 + \sigma)^{-1}$}. Recall that $L = L_1^2 + 2 - \frac{\mu}{2}$ and the desired result follows.
\end{proof}

It follows from~\cite[Theorem 2.2]{LPI_GD_alg} that in order to obtain a supremum norm interpolation $\delta$ using local polynomial interpolation, it is necessary that the size of the interpolation grid is $m = O( C_{\mu,L_1,L_2}(d,\eta)(1/\delta)^{1/\eta})$. Thus, $m = O( C_{\mu,L_1,L_2}(d,\eta)( \sqrt{p}(1+\sigma)/\varepsilon)^{1/\eta})$. Finally, at every iteration, in order to guarantee  $\delta \leq  \frac{\varepsilon}{2\sqrt{p}}(1 + \sigma)^{-1}$, we require a grid of size
\begin{align*}
    m^d = O(C_{\mu,L_1,L_2}(d,\eta)(\sqrt{p}(1+\sigma)/\varepsilon)^{d/\eta}),
\end{align*}
with the same number of oracle computations per iteration. Thus, we can conclude that the total number of gradient computations required by Algorithm~\ref{alg:nesterov_acc} to reach some desired accuracy $\varepsilon$ is 
\begin{align*}
    O\left(\sqrt{L_1\sigma}C_{\mu,L_1,L_2}(d,\eta)\left({1}/{\varepsilon} \right)^{d/\eta}\log\left({1}/{\varepsilon}\right)   \right).
\end{align*}




The following section shows the first empirical study of local polynomial interpolation-based gradient methods and their comparisons with their accelerated variants.

\section{Preliminary Numerical Analysis}\label{sec:num}

For our experimental setup, we follow, as closely as possible, the theoretical setup described in~\cite{LPI_GD_alg}. We study the linear regression for the case where $d = 1$, $h = 0.01$, and $n = 1000$. Then, we have $\mathcal{D} = \{x^{(i)} \in [0.01, 0.99] : i \in [1000] \}$. The one-dimensional linear regression problem can be described as 
\begin{equation*}
    \min_{w \in \mathbb{R}} \sum_{i=1}^{n} \lVert wx_i - y_i \rVert^2
\end{equation*}
where $x_i \in \mathcal{D}$. $y_i$ is found by first randomly generating an arbitrary $w_{\text{init}}$ from $(0,1)$. Then, $y_i = w_{\text{init}} x_i + \mathcal{N}(0, 0.05)$. This is a rather simple ERM problem to be solved. However, to our knowledge, there has been no prior implementation or empirical study of the weight interpolation calculation and LPI-GD algorithm described in \cite{LPI_GD_alg}. This allow us to focus on the showing numerical evidence for the acceleration of the proposed method.

Initially, we tested LPI-GD with different grid cardinalities to empirically determine an appropriate grid cardinality to solve the aforementioned linear regression problem. Since $n = 1000$, we tested the following grid cardinalities $ \lvert \mathcal{G}_m \rvert \in \{ 100, 200, 250, 500, 800\}$. The interpolation weights for the various grids were calculated following the procedure mentioned in \cite{LPI_GD_alg}. However, we used a step size of 1.0 as we experimentally found this to give the best convergence results. While running our experiments, we identified that the initialization point of the algorithm plays a crucial role in the smooth convergence to an approximate solution. We tested various graph sizes and a range of initialization points with this observation in mind. Since $w_{\text{init}} \in (0, 1)$, we tested various initial points $w_0 \in \{0, 0.2, 0.4, 0.6, 0.8, 1.0\}$. We then ran a grid search over the grid cardinalities $\lvert \mathcal{G}_m \rvert$ and the initial point $w_0$ to analyze the effects of these hyperparameters.
These experiments helped us determine the trade-off between the accuracy, number of iterations to converge, and oracle complexity. Note that the oracle complexity of LPI-GD methods is equivalent to the cardinality of the grid.
\par
Using the results from our grid search, we were able to identify the initialization point that provided the best computational performance for the different graph cardinalities we tested.
\begin{figure}[t!]
    \centering
    \includegraphics[width=0.9\columnwidth]{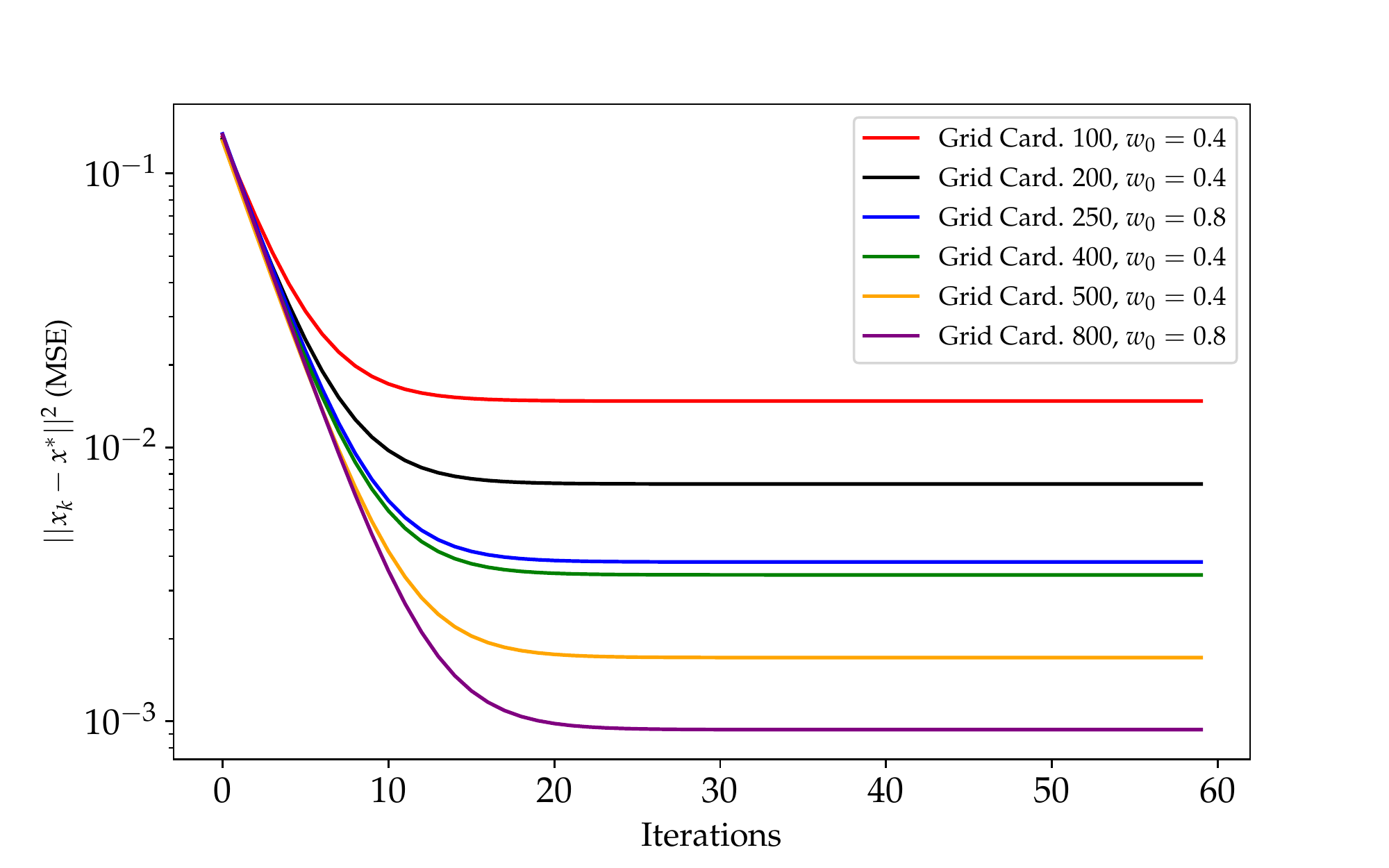}
    \caption{\textbf{Performance of LPI-GD for various graph cardinalities and initializations.} Minimum Squared Error (MSE) achieved by LPI-GD for different values of grid cardinality and initialization points. As the cardinality of the grid increases, a lower MSE is obtained for the same number of iterations. There exists a dependency on the initialization point, but its explicit characterization is not yet fully explained by our analysis.}
    \label{fig:Compare_Graph_Card}
\end{figure}
Figure~\ref{fig:Compare_Graph_Card} shows that there is a clear trend between increasing the graph cardinality and the accuracy of the solution. This is expected considering that a larger graph cardinality would correlate with a more accurate estimation of the full gradient. 
\par 
We then compare LPI-GD with SGD and GD. Our comparison uses the following hyperparameters for each method. For  LPI-GD, we use a graph cardinality of 500, initial point 0.4, and step-size 1.0. For SGD, we use a batch size of 500 and step-size 1.0. For GD, we use a step size of 1.0 as well. These may not be the optimal parameters, but they are the parameters we chose in order to have a ``fair" comparison. ``Fair" in this case means that we allow LPI-GD and SGD to make the same number of oracle calls per iterations and that we do not optimize the step-size for LPI-GD, SGD, and GD but rather use the same step-size which allows all the algorithms to converge well in practice. Figure~\ref{fig:Compare_GD_methods} shows that LPI-GD and SGD underperform with respect to GD in terms of the final accuracy and the convergence rate. However, we note that LPI-GD and SGD use half the oracle calls per iteration for our chosen parameters compared to GD. In addition, LPI-GD and SGD have a comparable error, with LPI-GD converging slightly faster.

\begin{figure}[t!]
\centering
\begin{subfigure}{\columnwidth}
\includegraphics[width=0.9\columnwidth]{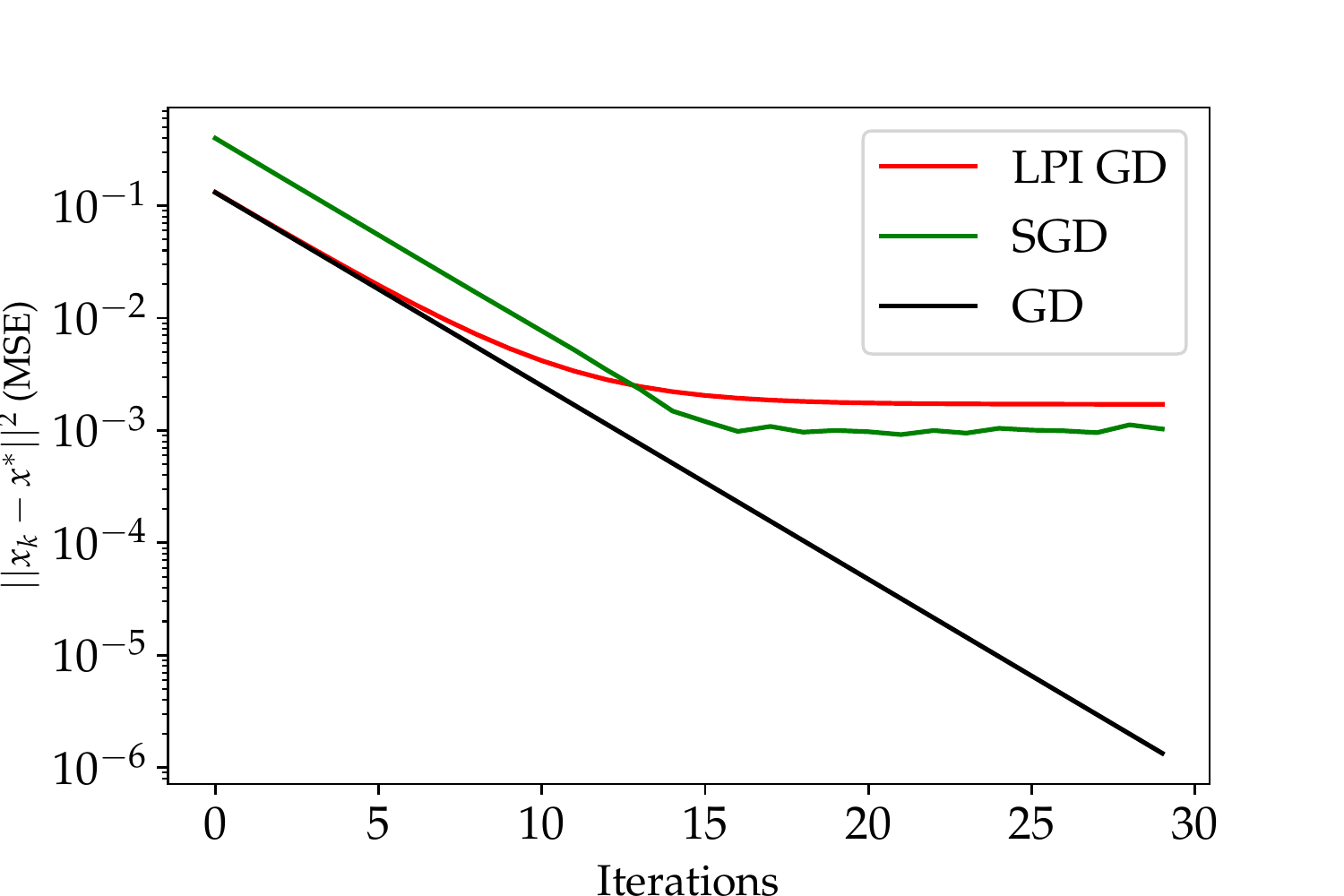}%
\caption{\textbf{Iteration Complexity of LPI-GD, SGD, and GD.} The Mean Square Error (MSE) of LPI-GD, SGD, and GD as the number of iterations increases. GD has linear convergence rates as predicted by the theory. SGD and LPI-GD have similar iteration complexity.}
\end{subfigure}\\
\begin{subfigure}{\columnwidth}
\includegraphics[width=0.9\columnwidth]{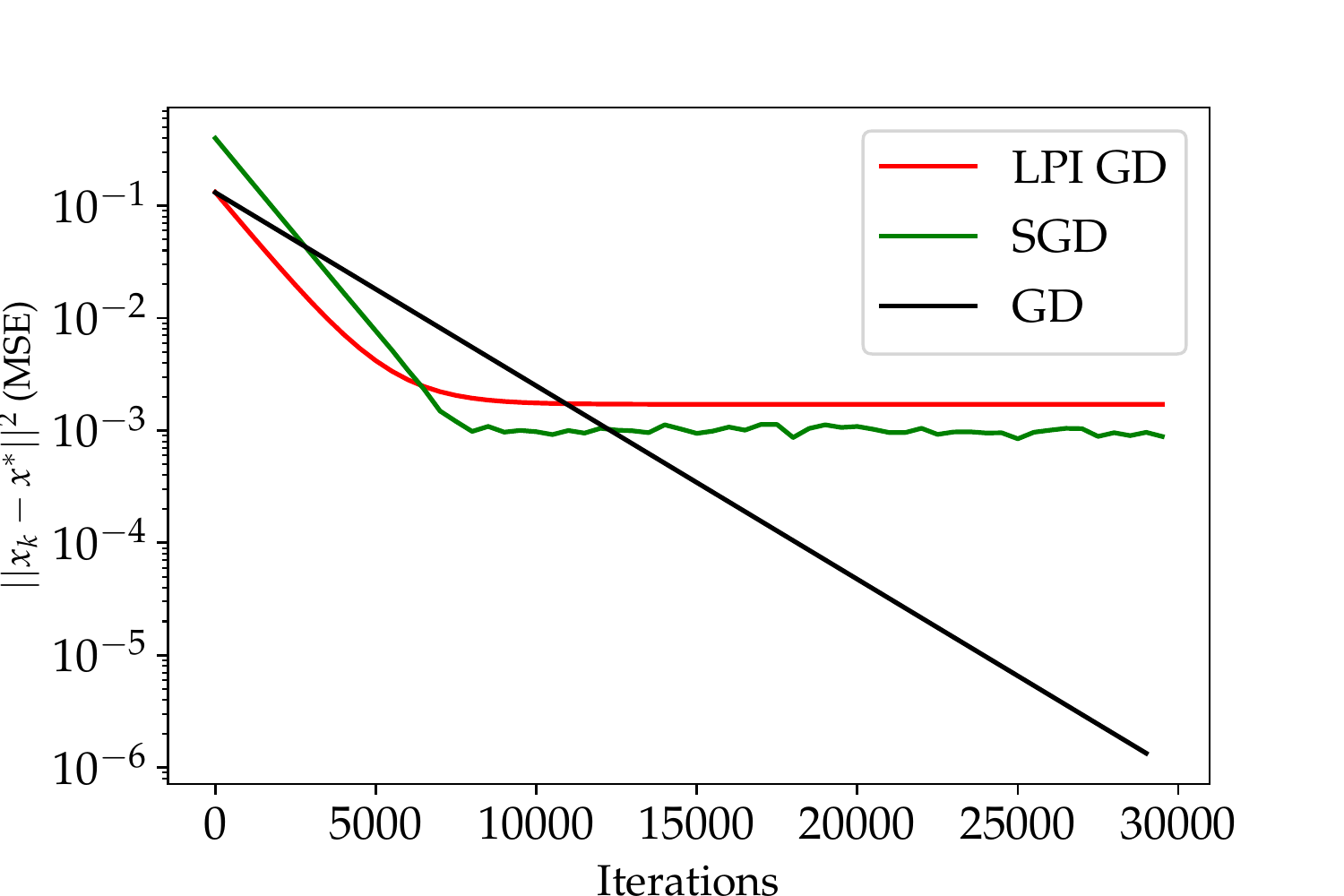}%
\caption{\textbf{Oracle Complexity of LPI-GD, SGD, and GD.} The Mean Square Error (MSE) of LPI-GD, SGD, and GD as the number of oracle calls increases. GD has linear convergence rates as predicted by the theory. SGD and LPI-GD have similar iteration complexity but both outperform GD in the initial iterations.}
\label{fig:sub3} \hfill
\end{subfigure}%
\caption{\textbf{Comparison of Iteration and Oracle complexities of LPI-GD, SGD, and GD.}}
\label{fig:Compare_GD_methods}
\end{figure}
\par
Finally, we will compare LPI-GD to Algorithm~\ref{alg:catalyst_acc}, and Algorithm \ref{alg:nesterov_acc}. We use the same hyperparameters for LPI-GD. For Algorithm~\ref{alg:catalyst_acc}, we use the same grid found for LPI-GD and a parameter $\beta = 0.99$. For Algorithm \ref{alg:nesterov_acc}, we use the same grid found for LPI-GD and a momentum parameter of 0.2. All of the initialization points are $w_0 = 0.4$. The results of this comparison can be seen in Figure \ref{fig:Compare_New_methods}. Unfortunately, Algorithm~\ref{alg:catalyst_acc} did not demonstrate the desired speedup. We hypothesize that the linear regression problem is not sufficiently sophisticated for us to see a speedup. However, Algorithm \ref{alg:nesterov_acc} has a faster convergence rate in comparison to LPI-GD. Despite the difference in convergence rates, all the compared algorithms converge to the same level of accuracy.

\begin{figure}[t!]
    \centering
    \includegraphics[width=0.9\columnwidth]{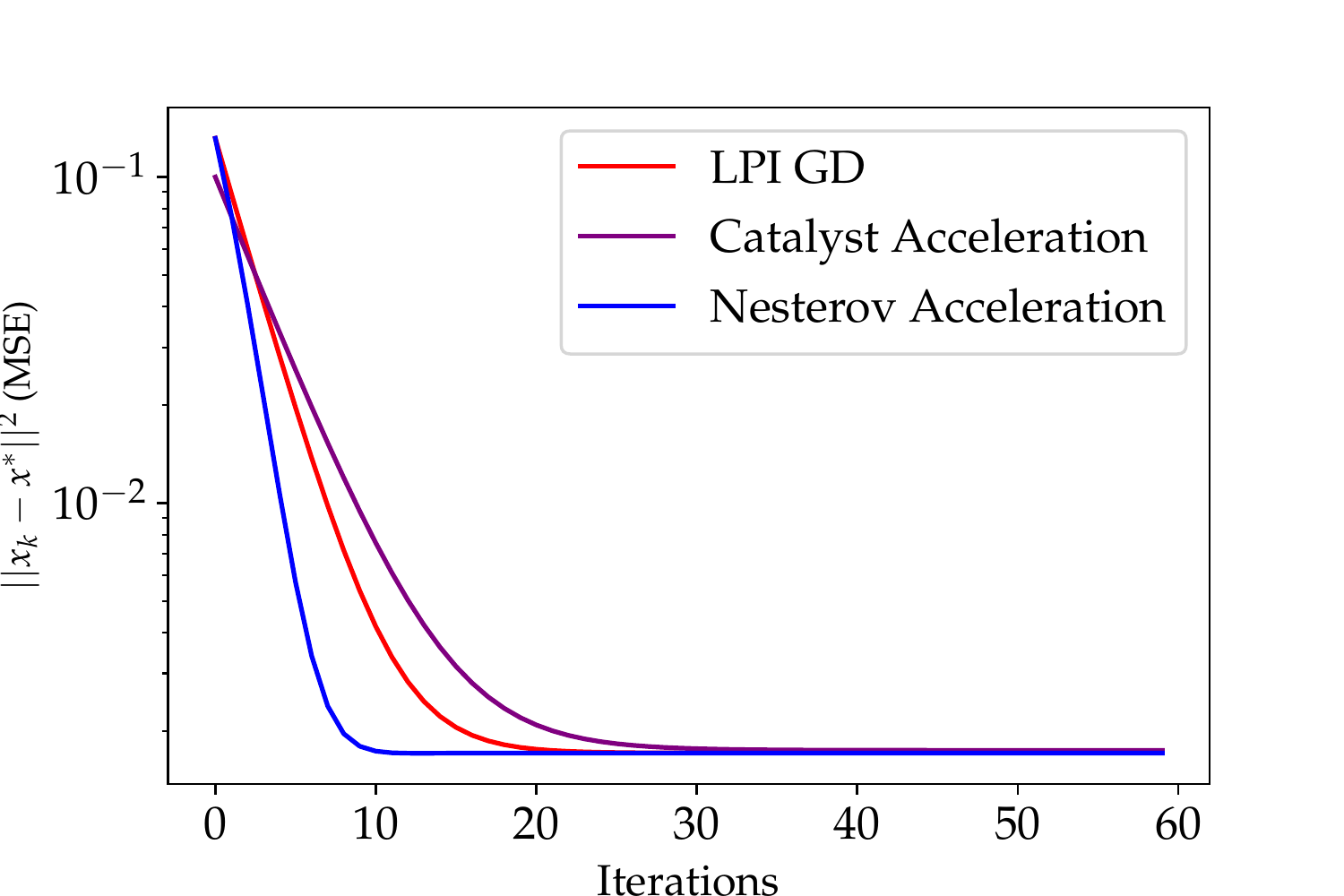}
    \caption{\textbf{Iteration Complexity of LPI-GD, Algorithm~\ref{alg:catalyst_acc}, and Algorithm~\ref{alg:nesterov_acc}.} Acceleration is achieved by Algorithm~\ref{alg:nesterov_acc} in comparison with LPI-GD. Algorithm~\ref{alg:catalyst_acc} does not show empirical evidence of better performance than GD for this particular problem class. }
    \label{fig:Compare_New_methods}
\end{figure}

\section{Discussion on Empirical Results}\label{sec:diss}

This section will highlight some key findings regarding the computational implementation of LPI-GD and its accelerated variants. The main overhead of the LPI-GD, Algorithm~\ref{alg:catalyst_acc}, and Algorithm \ref{alg:nesterov_acc} comes from calculating interpolation weights. There were multiple sources for the increased computational time in our implementation. One of the primary bottlenecks in the fact that we have to compute $\lvert x \rvert$ matrix inversions for matrices of dimension $\binom{l + d}{d} \times \binom{l + d}{d}$ where $\lvert x \rvert$ is the number of data points, $l$ is the order of the objective function minus $1$, and $d$ is the dimension of the data. Unfortunately, these matrix inversions are unavoidable as they are key components in calculating the interpolation weights. Another major source of the slowdown is that we could not determine a good method to vectorize weight computations. There should be a total of $y \times x$ weights. Our current implementation requires a total of $y\times x$ iterations to calculate the weights as we compute the weights individually. A procedure to calculate all the weights simultaneously could greatly improve the efficiency of future implementations. 
\par
Additionally, we observed that LPI-GD, Algorithm~\ref{alg:catalyst_acc}, and Algorithm \ref{alg:nesterov_acc} heavily rely on a good choice of initialization for fast convergence and that the choice of grid cardinality is integral to increasing the accuracy of the solutions. The choice of grid cardinality is essential as it is directly related to the accuracy of the approximation of the true gradient. However, increasing the grid cardinality beyond the number of data points defeats the original purpose of using polynomial interpolations. In that case, GD should be used instead, as we would be using a larger number of oracle calls per iteration than GD. Thus, it is crucial to determine the grid cardinality-accuracy trade-off to determine whether LPI methods should be used.

\section{Conclusions and Future Work} \label{sec:conclusion}

We proposed two acceleration approaches for the local polynomial interpolation-based gradient descent method LPI-GD for large-scale ERM problems. The first accelerated scheme is based on the Catalyst technique, and the second is based on fast methods with inexact oracles. We showed an improvement in the oracle complexities of the proposed methods with respect to the function condition number. We also provide the first empirical study of LPI-GD. While theoretical acceleration is shown for the two proposed schemes, only the inexact oracle generated observable faster convergence rates. We conjecture that Catalyst approaches would perform better for other function classes and parameter regimes.

Future empirical work would extend the complexity of problems we solve by choosing more complex problems and increasing the dimensions of our parameters and data. It is unclear how the accuracy of the gradient approximation is affected when the data dimension increases. To our knowledge, there is no existing library or methods that calculate the interpolation weights according to \cite{LPI_GD_alg} efficiently. Developing these methods efficiently should enable a broader study of LPI methods.

\section{Acknowledgements} \label{sec:ack}

The authors would like to thank Anuran Makur for introducing the authors to LPI methods and subsequent fruitful discussions.
	
\bibliographystyle{IEEEtran} 
\bibliography{ref}

\end{document}